\documentclass[12pt]{amsart}
\usepackage{amsfonts, amssymb, latexsym, array}
\usepackage{fullpage,color}
\usepackage{verbatim,euscript,longtable}

\numberwithin{equation}{section}
\usepackage[colorlinks=true,linkcolor=blue,urlcolor=violet,citecolor=magenta]{hyperref}

\input {cyracc.def}
\font\tencyr=wncyr10 
\font\tencyi=wncyi10 
\font\tencysc=wncysc10 
\def\rus{\tencyr\cyracc}
\def\rusi{\tencyi\cyracc}
\def\rusc{\tencysc\cyracc}

\newtheorem{thm}{Theorem}[section]
\newtheorem{lm}[thm]{Lemma}
\newtheorem{cl}[thm]{Corollary}
\newtheorem{prop}[thm]{Proposition}

\theoremstyle{remark}
\newtheorem{rmk}[thm]{Remark}

\theoremstyle{definition}
\newtheorem{ex}[thm]{Example}
\newtheorem{df}{Definition}

\newcommand {\g}{{\mathfrak g}}

\newcommand {\q}{{\mathfrak q}}

\newcommand {\ca}{{\mathcal A}}
\newcommand {\cb}{{\mathcal B}}

\newcommand {\cf}{{\mathcal F}}

\newcommand {\co}{{\mathcal O}}

\newcommand {\BV}{{\mathbb V}}
\newcommand {\BW}{{\mathbb W}}

\newcommand {\BQ}{{\mathbb Q}}
\newcommand {\BZ}{{\mathbb Z}}
\newcommand {\BN}{{\mathbb N}}

\newcommand {\sfr}{\mathsf R}

\newcommand {\md}{/\!\!/}

\newcommand{\lb}{\lambda}
\newcommand{\ap}{\alpha}

\renewcommand{\le}{\leqslant}
\renewcommand{\ge}{\geqslant}

\newcommand{\curle}{\preccurlyeq}

\newcommand{\eus}{\EuScript}

\newcommand {\codim}{{\mathrm{codim}}}

\newcommand {\ed}{{\mathsf{ed}}}

\newcommand {\hd}{{\mathsf{hd}}}

\newcommand {\Lie}{{\mathrm{Lie\,}}}

\newcommand {\rk}{{\mathsf{rk\,}}}

\newcommand {\spe}{{\mathsf{Spec}}}

\newcommand {\trdeg}{{\mathrm{trdeg\,}}}

\newcommand {\sltri}{\mathfrak{sl}_3}
\newcommand {\GR}[2]{{\textrm{{\bf #1}}}_{#2}}

\newcommand {\ov}{\overline}
\newcommand {\un}{\underline}

\newfam\eusfam%
\font\Bbbfont=msbm10 scaled 1200%
\def\bbk{\hbox {\Bbbfont\char'174}}

\begin{document}
\setlength{\parskip}{3pt plus 5pt minus 0pt}
\hfill { {\scriptsize December 7, 2008}}
\vskip1ex

\title[On the derived group of a maximal unipotent subgroup]
{Actions of the derived group of a maximal unipotent subgroup on $G$-varieties}
\author[D.\,Panyushev]{Dmitri I.~Panyushev}
\address[]{Independent University of Moscow,
Bol'shoi Vlasevskii per. 11, 119002 Moscow, \ Russia
\hfil\break\indent
Institute for Information Transmission Problems, B. Karetnyi per. 19, Moscow 127994
}
\email{panyush@mccme.ru}
\urladdr{http://www.mccme.ru/~panyush}
\maketitle

\section*{Introduction}
\noindent
The ground field $\bbk$ is algebraically closed and of characteristic zero.
Let $G$ be a semisimple simply-connected algebraic group over $\bbk$ and 
$U$ a maximal unipotent subgroup of $G$.
One of the fundamental invariant-theoretic facts, which goes back to 
Had\v{z}iev~\cite{had}, is that
$\bbk[G/U]$ is a finitely generated $\bbk$-algebra and regarded as $G$-module
it contains every finite-dimensional simple $G$-module exactly once.
From this, one readily deduces that the algebra of $U$-invariants, 
$\bbk[G/U]^U$, is polynomial. 
More precisely, choose a maximal torus $T\subset \mathsf{Norm}_G(U)$. 
Let $r$ be the rank of $G$,
$\varpi_1,\dots,\varpi_r$ the fundamental weights of $T$ corresponding to $U$, and 
$\ap_1,\dots,\ap_r$ the respective simple roots.
Set $\mathfrak X_+=\sum_{i=1}^r \BN\varpi_i$, and let
$\mathsf R(\lb)$ denote the simple $G$-module with highest weight $\lb\in\mathfrak X_+$.
Then 
\[
     \bbk[G/U]\simeq \bigoplus_{\lb\in\mathfrak X_+}\sfr(\lb).
\]
Let $f_i$ be a non-zero element of one-dimensional space $\mathsf R(\varpi_i)^U
\subset \bbk[G/U]^U$.
Then $\bbk[G/U]^U$ is freely generated by $f_1,\dots,f_r$.

For an affine $G$-variety $X$, the algebra of $U$-invariants, $\bbk[X]^U$,
is multigraded (by $T$-weights). If $X=V$ is a $G$-module, then there is an integral 
formula for the corresponding Poincar\'e series 
\cite[Theorem\,1]{br83}. Using that formula,
M.\,Brion discovered useful ``symmetries'' of the Poincar\'e series and applied them
(in case $G$ is simple)
to obtaining the classification of simple $G$-modules
with polynomial algebras  $\bbk[V]^U$
\cite[Ch.\,III]{br83}. Afterwards, I proved that similar ``symmetries'' of Poincar\'e
series occur for conical factorial $G$-varieties with only rational singularities
\cite{pa95}, \cite[Ch.\,5]{pa99}. Since there is no integral formula for Poincar\'e series in general, another technique was employed.  Namely, I used the transfer principle for $U$, 
``symmetries''  of the Poincar\'e series of $\bbk[G/U]$, and results of F.\,Knop relating 
the canonical module of an algebra and a subalgebra of invariants \cite{kn86}.

Our objective is to extend  these results to the derived group $U'=(U,U)$.
In Section~\ref{sect:inv}, we prove that 
$\sfr(\lb)^{U'}$ is a cyclic $U/U'$-module for any $\lb\in\mathfrak X_+$
and 
$\dim \sfr(\lb)^{U'}=\prod_{i=1}^r((\lb, \ap_i^\vee)+1)$, where $\ap_i^\vee=2\ap_i/(\ap_i,\ap_i)$,
see Theorem~\ref{thm:main1}. From these properties, we 
deduce that $\bbk[G/U]^{U'}$ is a polynomial algebra of Krull 
dimension $2r$. More precisely,  we have 
$\dim\sfr(\varpi_i)^{U'}=2$ for each
$i$, and if $(f_i,\tilde f_i)$ is a basis in $\sfr(\varpi_i)^{U'}$, then
$\{f_i,\tilde f_i\mid i=1,\dots,r\}$ freely generate $\bbk[G/U]^{U'}$ 
(see Theorem~\ref{thm:main2}).
This fact seems to have remained unnoticed before. As a by-product, we show that 
the subgroup $TU'\subset G$ is epimorphic (i.e., $\bbk[G]^{TU'}=\bbk$) if and only if
$G\ne SL_2, SL_3$.

Section~\ref{sect:gen_prop} is devoted to general properties of  $U'$-actions on
affine $G$-varieties.  We show that  $\bbk[G/U']$ is generated by fundamental
$G$-modules sitting in it, and using this fact we explicitly construct an equivariant affine embedding
of $G/U'$ with the boundary of codimension $\ge 2$ (Theorem~\ref{thm:model}).
Since
$\bbk[G/U']$ is finitely generated, $\bbk[X]^{U'}$ is finitely generated for any affine $G$-variety $X$ \cite{gr}. 
Furthermore,
$\spe(\bbk[X]^{U'})$ inherits some other good properties of $X$ (factoriality, rationality of
singularities) (Theorem~\ref{thm:RS}).
We also give an algorithm for constructing a finite generating system of
$\bbk[X]^{U'}$, if generators of $\bbk[X]^U$ are already known (Theorem~\ref{thm:base-X-U'}). 
This appears to be very helpful in classifying simple $G$-modules with polynomial algebras
of $U'$-invariants (for $G$  simple).

In Section~\ref{sect:poinc}, we study the Poincar\'e series of multigraded algebras
$\bbk[X]^{U'}$, where $X$ is factorial affine $G$-variety with only rational singularities (e.g.
$X$ can be a $G$-module). Assuming that $G\ne SL_2, SL_3$, 
we obtain analogues of our results for Poincar\'e series of $\bbk[X]^U$.
One of the practical outcomes concerns the case in which
$V$ is a $G$-module and $\bbk[V]^{U'}$ is polynomial. If $d_1,\dots, d_m$
(resp. $\mu_1,\dots,\mu_m$) are the degrees (resp. $T$-weights)
of basic  $U'$-invariants, then $\sum_i d_i\le \dim V$ and 
$\sum_i \mu_i \le 2\rho-\sum_{j=1}^r \ap_j$, where $\rho=\sum_{j=1}^r \varpi_j$. 
The second inequality requires some explanations, though. Unlike the case of 
$U$-invariants, there is no natural free monoid
containing the $T$-weights of all $U'$-invariants. But for $G\ne SL_2, SL_3$, these $T$-weights generate a convex cone. Therefore,
such a free monoid does exist, and  the above inequality for $\sum_i \mu_i$
is understood as componentwise inequality with respect to any such monoid and its basis.
Moreover,  $\sum_i d_i= \dim V$ if and only if  $\sum_i \mu_i = 2\rho-\sum_{j=1}^r \ap_j$.
Again, these relations are to be useful for our classification of polynomial algebras
$\bbk[V]^{U'}$, which is obtained in Section~\ref{sect:classif}. Note that $2\rho-\sum_{j=1}^r \ap_j$
is the sum of all positive non-simple roots, i.e., the roots of $U'$.

Section~\ref{sect:combin} is a kind of combinatorial digression. 
Let $\mathcal C$ be the cone generated by all $T$-weights occurring in $\bbk[G/U]^{U'}$.
Our description of generators shows that $\mathcal C$ is actually generated by 
$\varpi_i, \varpi_i-\ap_i$ ($i=1,\dots,r$). We prove that the dual cone of $\mathcal C$ is 
generated by the  non-simple positive roots (Theorem~\ref{thm:dual-cone}). 
We also obtain  
a partition of $\mathcal C$ in simplicial cones, which is parametrised by the disjoint subsets on the Dynkin diagram of $G$.

My motivation to consider $U'$-invariants arose from attempts to understand the structure of centralisers of certain nilpotent elements in simple Lie algebras. For applications to centralisers one needs Theorem~\ref{thm:main1} in case of $SL_3$, and this was the result
initially proved. This application will be the subject of a subsequent article.

\noindent
\un{\sl Notation}.
If an algebraic group $Q$ acts on an irreducible affine variety $X$, then 

\textbullet \ $Q_x=\{q\in Q\mid q{\cdot}x=x\}$ is the stabiliser of $x\in X$;

\textbullet \ $\bbk[X]^Q$ is the algebra of $Q$-invariant polynomial functions on $X$.
If $\bbk[X]^Q$ is finitely generated, then $X\md Q:=\spe(\bbk[X]^Q)$, and
the {\it quotient morphism\/} $\pi_{X,Q}: X\to X\md Q$ is the mapping associated with
the embedding $\bbk[X]^Q \hookrightarrow \bbk[X]$.

\textbullet  \ $\bbk(X)^Q$ is the field of $Q$-invariant rational functions;

\noindent
Throughout, $G$ is a semisimple simply-connected algebraic group and $r=\rk G$.
\\ \indent
{\bf --}  $\Delta$ is the root system of $(G,T)$, 
$\Pi=\{\ap_1,\dots,\ap_r\}$ are the simple roots corresponding to $U$, and 
$\varpi_1,\dots,\varpi_r$ are the corresponding fundamental weights. 
\\ \indent
{\bf --}  The character group of $T$ is denoted by $\mathfrak X$. All roots and weights are regarded as elements of 
the $r$-dimensional vector space $\mathfrak X\otimes \BQ=:\mathfrak X_\BQ$.
For any $\lb\in\mathfrak X_+$, $\lb^*$ is the highest weight of the dual $G$-module.
The $\mu$-weight space of $\mathsf R(\lb)$ is denoted by $\sfr(\lb)_\mu$.

\vskip1ex
{\small {\bf Acknowledgements.}  
This work was done during my  stay at the Max-Planck-Institut f\"ur Mathematik (Bonn). 
I am grateful to this institution for the warm hospitality and support.}

\section{The algebra of $U'$-invariants on $G/U$} 
\label{sect:inv}

\noindent
For any $\lb\in\mathfrak X_+$, we wish to study the subspace $\sfr(\lb)^{U'}$.
First of all, we notice that $B\subset \mathsf{Norm}_G(U')$ 
(actually, they are equal if $G$ has no simple factors $SL_2$) and therefore $\sfr(\lb)^{U'}$ is a $B/U'$-module.
In particular, 
$T$ normalises $U'$ and hence $\sfr(\lb)^{U'}$ is a direct sum of its own weight spaces.
Let $\eus P(\lb)$ be the set of weights of $\sfr(\lb)$.  It is a poset with respect to the
{\it root order}. This means that $\mu$ {covers\/} $\nu$ if $\mu-\nu\in\Pi$.
Then $\lb$ is the unique maximal element of $\eus P(\lb)$.
Let $e_i\in\g=\Lie(G)$ be a root vector corresponding to $\ap_i\in\Pi$.

Given a nonzero $x\in \sfr(\lb)^{U'}$, consider
\[
  M_x=\{(n_1,\dots,n_r)\in \BN^r\mid e_1^{n_1}\dots e_r^{n_r}(x)\ne 0\} .
\]
We also write $\boldsymbol n=(n_1,\dots,n_r)$ and $\boldsymbol e^{\boldsymbol n}=
e_1^{n_1}\dots e_r^{n_r}$. Notice that $\boldsymbol e^{\boldsymbol n}(x)$ does not
depend on the ordering of $e_i$'s since 
$[e_i,e_j]\in \Lie(U')$  for all $i,j$ and $\sfr(\lb)^{U'}$ is an $U/U'$-module. 
We regard $M_x$ as poset with respect to the
componentwise inequalities, i.e., 
$\boldsymbol n \succcurlyeq \boldsymbol {n'}$ if and only if $n_i\ge n'_i$ for all $i$.
Clearly, $M_x$ is finite and $(0,\dots,0)$ is the unique minimal element of it.

\begin{lm}   \label{lm:unique-max}
Let $x\in\sfr(\lb)^{U'}$ be a weight vector.
The poset $M_x$ contains a unique maximal element, say $\boldsymbol m=
(m_1,\dots,m_r)$. Furthermore, 
$\boldsymbol e^{\boldsymbol m}(x)$ is a highest vector of\/ $\sfr(\lb)$.
\end{lm}\begin{proof}
If ${\boldsymbol n}\in M_x$ is maximal, then $e_i(\boldsymbol e^{\boldsymbol n}(x))=0$ for each
$i$. Hence $\boldsymbol e^{\boldsymbol n}(x)$ is a highest vector of $\sfr(\lb)$.
Next, 
\[
  \textrm{the weight of }\boldsymbol e^{\boldsymbol n}(x)=
  (\textrm{the weight of }x)+\sum_{i=1}^r n_i\ap_i \ .
\]
Hence all nonzero vectors of the form $\boldsymbol e^{\boldsymbol n}(x)$ are 
linearly independent. This yields the uniqueness of a maximal element.
\end{proof}

\begin{cl}
$M_x$ is a multi-dimensional array, i.e., $M_x=\{(n_1,\dots,n_r)\mid 0\le n_i\le m_i
\ \forall i\}$.
\end{cl}

\noindent
Let $I_\lb$ denote the set of $T$-weights in $\sfr(\lb)^{U'}$. It is a subset of $\eus P(\lb)$.
\begin{prop}   \label{prop:free}
For any $\lb\in \mathfrak X_+$, $\sfr(\lb)^{U'}$ is a multiplicity free $T$-module.
More precisely, 
\[
   \sfr(\lb)^{U'}=\bigoplus_{\mu\in I_\lb} \sfr(\lb)^{U'}_\mu\ ,
\]
where $\dim \sfr(\lb)^{U'}_\mu=1$ for each $\mu$ and $I_\lb\subset \{\lb-\sum_ia_i\ap_i \mid
0\le a_i\le (\lb,\ap_i^\vee)\}$.
\end{prop}\begin{proof}
If $x\in \sfr(\lb)^{U'}_\mu$ and $(m_1,\dots,m_r)$ is the maximal element of $M_x$, 
then $\mu+\sum_i m_i\ap_i=\lb$ and $\mu+\sum_i n_i\ap_i\in \eus P(\lb)$ for any
$(n_1,\dots,n_r)\in M_x$. In particular, $\lb-m_i\ap_i\in\eus P(\lb)$. Whence
$m_i\le (\lb,\ap_i^\vee)$ and $I_\lb\subset \{\lb-\sum_ia_i\ap_i \mid
0\le a_i\le (\lb,\ap_i^\vee)\}$.

Assume that $x,y\in \sfr(\lb)^{U'}_\mu$ are linearly independent. It follows from 
Lemma~\ref{lm:unique-max} that $M_x=M_y$. Since $\boldsymbol e^{\boldsymbol m}(x),
\boldsymbol e^{\boldsymbol m}(y)\in \sfr(\lb)_\lb$, we have 
$\boldsymbol e^{\boldsymbol m}(x-cy)=0$ for some $c\in\bbk^\times$. This means that
$M_{x-cy}\ne M_x$, a contradiction! Thus, each $\sfr(\lb)^{U'}_\mu$ is one-dimensional.
\end{proof}

\begin{lm}  \label{lm:connected}
$I_\lb$ is a connected subset
in the Hasse diagram of $\eus P(\lb)$ that contains $\lb$.
\end{lm}\begin{proof}
Indeed, suppose $0\ne v\in \sfr(\lb)^{U'}_\mu$. If $e_{\ap_i}{\cdot}v=0$ for all $i$, then
$v$ is a $U$-invariant and hence $\mu=\lb$. Otherwise, we have
$e_{\ap_i}{\cdot}v\ne 0$ for some $i$
and therefore $\mu+\ap_i$ is also a weight of $\sfr(\lb)^{U'}$. Then we argue by induction.
\end{proof}

\begin{prop}         \label{fundam}
For any fundamental weight $\varpi_i$, we have
$\sfr(\varpi_i)^{U'}=\sfr(\varpi_i)_{\varpi_i}\oplus\sfr(\varpi_i)_{\varpi_i-\ap_i}$.
In particular, $I_{\varpi_i}=\{\varpi_i,\varpi_i-\ap_i\}$ and $\dim\sfr(\varpi_i)^{U'}=2$.
\end{prop}
\begin{proof} 
Note that $\varpi_i-\ap_i\in \eus P(\varpi_i)$ 
and $\dim \sfr(\varpi_i)_{\varpi_i-\ap_i}=1$,
while  $\varpi_i-2\ap_i\not\in \eus P(\varpi_i)$.
We obviously have 
$\sfr(\varpi_i)^{U'}\supset \sfr(\varpi_i)_{\varpi_i}\oplus\sfr(\varpi_i)_{\varpi_i-\ap_i}$.
Any weight of $\sfr(\varpi_i)$ covered by $\varpi_i-\ap_i$ is of the form
$\varpi_i-\ap_i-\ap_j$, where $\ap_j$ is a simple root adjacent to $\ap_i$ in the
Dynkin diagram of $G$. Since $\varpi_i-\ap_j\not\in\eus P(\varpi_i)$,  
Kostant's weight multiplicity formula shows that
$\dim \sfr(\varpi_i)_{\varpi_i-\ap_i-\ap_j}=1$.
Since $\ap_i+\ap_j$ is a root of $U'$, we have
$\sfr(\varpi_i)_{\varpi_i-\ap_i-\ap_j}\not\subset \sfr(\varpi_i)^{U'}$ and  it follows from
Lemma~\ref{lm:connected}  that there cannot be anything else in $\sfr(\varpi_i)^{U'}$. 
\end{proof}

Set $\tilde X=\spe(\bbk[G]^U)$. It is an affine $G$-variety containing $G/U$ as a dense open subset.
Recall that $\tilde X$ has the following explicit model, see \cite{vp}.
Let $v_{-\varpi_i}$ be a lowest weight vector in $\sfr(\varpi_i)^*$. Then
the stabiliser of $(v_{-\varpi_1},\ldots,v_{-\varpi_r})\in \sfr(\varpi_1)^*\oplus\ldots\oplus\sfr(\varpi_r)^* $ is the maximal unipotent subgroup that is opposite to $U$ and
\[
\tilde X\simeq\ov{G{\cdot}
(v_{-\varpi_1},\ldots,v_{-\varpi_r})}\subset \sfr(\varpi_1)^*\oplus\ldots\oplus\sfr(\varpi_r)^* \ .
\]
Let $p_i: \tilde X\to  \sfr(\varpi_i)^*$ be the projection to the $i$-th component.
Then the pull-back of the linear functions on $ \sfr(\varpi_i)^*$ yields the unique copy
of the $G$-module $ \sfr(\varpi_i)$ in $\bbk[\tilde X]$. The additive decomposition 
$\bbk[\tilde X]= \bigoplus_{\lb\in\mathfrak X_+}\sfr(\lb)$ is a polygrading; i.e., 
if $f_i\in\sfr(\lb_i)\subset \bbk[\tilde X]$, $i=1,2$, then $f_1 f_2\in \sfr(\lb_1+\lb_2)$.

\begin{df}   \label{def:cyclic}
Let $Q$ be an algebraic group with Lie algebra $\q$. A $Q$-module $V$ 
is said to be {\it cyclic\/}  if there is $v\in V$ such that $\eus U(\q){\cdot}v=V$, where
$\eus U(\q)$ is the enveloping algebra of $\q$. Such $v$ is called a {\it cyclic vector}.
\end{df}

\begin{thm}  \label{thm:main1}
For any $\lb\in\mathfrak X_+$, we have 
\begin{itemize}
\item[(i)] \ $I_\lb=\{\lb-\sum_{i=1}^r a_i\ap_i \mid 0\le a_i\le (\lb,\ap_i^\vee)\}$;
\item[(ii)] \ $\sfr(\lb)^{U'}$ is a cyclic $U/U'$-module of
dimension $\prod_{i=1}^r((\lb,\ap_i^\vee)+1)$. Up to  a scalar multiple, there is a unique 
cyclic vector that is a $T$-eigenvector.
\end{itemize}
\end{thm}\begin{proof}
In view of Lemma~\ref{lm:unique-max} and 
Proposition~\ref{prop:free}, it suffices to prove that $\sfr(\lb)^{U'}$ contains
a vector of weight $\lb-\sum_{i=1}^r(\lb,\ap_i^\vee)\ap_i$.
This vector have to be cyclic, because applying the $e_i$'s  to it we obtain
weight vectors with all weights from 
$\{\lb-\sum_{i=1}^r a_i\ap_i \mid 0\le a_i\le (\lb,\ap_i^\vee)\}$, hence 
the whole of  $\sfr(\lb)^{U'}$.

Let  $\tilde f_i$ be a nonzero vector in one-dimensional space
$\sfr(\varpi_i)_{\varpi_i-\ap_i}$. 
Using the unique copy of $\sfr(\varpi_i)$
inside $\bbk[\tilde X]$, we regard  $\tilde f_i$ as $U'$-invariant polynomial function on $\tilde X$.
Take the product (monomial) $F:=\prod_{i=1}^r \tilde f_i^{(\lb,\ap_i^\vee)}\in \bbk[\tilde X]$. 
Since $\bbk[\tilde X]$ is a domain, $F\ne 0$. 
The multiplicative  structure of
$\bbk[\tilde X]$ shows that $F\in \sfr(\lb)^{U'}$ and the weight of $F$ equals
$\sum_{i=1}^r(\lb,\ap_i^\vee)(\varpi_i-\ap_i)=\lb-\sum_{i=1}^r(\lb,\ap_i^\vee)\ap_i$.
\end{proof}
\begin{rmk}
For the group $TU'\subset B$, we have $\dim TU'=\dim U$.
It is well known that $TU'$ is a spherical subgroup of $G$
(e.g. apply \cite[Prop.\,1.1]{br87}). The sphericity also follows
from the fact $\sfr(\lb)^{U'}$ is a multiplicity free $T$-module
(Proposition~\ref{prop:free}).
That 
$\sfr(\lb)^{U'}$ is a multiplicity free $T$-module follows also from \cite[Corollary\,8]{ho89}. 
However, we obtain the explicit description of the corresponding weights and the 
$U/U'$-module structure of $\sfr(\lb)^{U'}$.
\end{rmk}

\begin{thm}   \label{thm:main2}
Let $f_i$ (resp. $\tilde f_i$) be a nonzero vector in one-dimensional space
$\sfr(\varpi_i)_{\varpi_i}$ (resp. $\sfr(\varpi_i)_{\varpi_i-\ap_i}$). 
Then the  algebra of $U'$-invariants, $\bbk[G/U]^{U'}$, is  freely generated by
$f_1,\tilde f_1,\dots,f_r,\tilde f_r$.
\end{thm}\begin{proof}
It follows from (the proof of) Theorem~\ref{thm:main1} that
the monomials $\prod_{i=1}^r f_i^{c_i}\tilde f_i^{(\lb,\ap_i^\vee)-c_i}$, $0\le c_i\le (\lb,\ap_i^\vee)$,
form a basis for $\dim \sfr(\lb)^{U'}$ for each $\lb\in\mathfrak X_+$. 
Hence $\bbk[G/U]^{U'}$ is generated by $f_1,\tilde f_1,\dots,f_r,\tilde f_r$.
Since $U'$ is unipotent and $\dim(G/U) -\dim U'=2r$, the Krull dimension of $\bbk[G/U]^{U'}$ 
is at least $2r$. Hence there is no relations between the above generators. 
\end{proof}

\noindent
Recall that a closed subgroup $H\subset G$ is said to be {\it epimorphic\/} 
if $\bbk[G/H]=\bbk$ or, equivalently, $\sfr(\lb)^H=\{0\}$ unless $\lb=0$, see 
e.g. \cite[\S\,23B]{gr}.

\begin{prop}   \label{prop:epi} 
Suppose $G$ is simple. The subgroup $TU'$ is  epimorphic 
if and only if $G\ne SL_2$ or $SL_3$.
\end{prop}\begin{proof}
The case of $SL_2$ is obvious, so we assume that $r\ge 2$.
In view of Theorem~\ref{thm:main2}, we have to check that neither of the monomials
$\prod_{i=1}^r (f_i^{c_i}\tilde f_i^{(\lb,\ap_i^\vee)-c_i})$, $0\le c_i\le (\lb,\ap_i^\vee)$, has zero weight if  $G\ne SL_3$.
The weight in question equals
\[
    \mu:=\sum_{i=1}^r  (\lb,\ap_i^\vee)\varpi_i - \sum_{i=1}^r  c_i\ap_i \ .
\] 
Set $\rho^\vee=\frac{1}{2}\underset{\gamma\in\Delta^+}{\sum}\gamma^\vee$.
Then $(\mu,\rho^\vee)=\sum_{i=1}^r  (\lb,\ap_i^\vee)(\varpi_i,\rho^\vee)-\sum_{i=1}^r  c_i$.
Notice that  
\[
  2(\varpi_i,\rho^\vee)=\sum_{\gamma\in\Delta^+}(\varpi_i,\gamma^\vee)
  \ge \#\{ \gamma\in \Delta^+\mid (\gamma,\varpi_i)>0\} .
\]
That is, $2(\varpi_i,\rho^\vee)$ is at least the dimension of the nilpotent
radical of the maximal parabolic subalgebra corresponding to $\varpi_i$.
This readily implies that $(\varpi_i,\rho^\vee)>1$ for all $i$ whenever $\g\ne\sltri$.
Whence $(\mu,\rho^\vee)$ is positive. 
\\ \indent
For $SL_3$, the monomial $(\tilde f_1\tilde f_2)^a$ has zero weight. That is,
$\sfr(a(\varpi_1+\varpi_2))^{TU'}\ne \{0\}$.
\end{proof}

\begin{rmk}
If $G=SL_3$, then $TU'$ is a Borel subgroup of a reductive subgroup 
$GL_2\subset SL_3$. 
Proposition~\ref{prop:epi} can also be deduced from a result of Pommerening
\cite[Korollar\,3.6]{pomm}.
\end{rmk}
\begin{ex} 
Let $U_n$ be a maximal unipotent subgroup of $G=SL_{n}$ and let $U_{n-1}$ be 
a maximal unipotent subgroup of a
standardly embedded group $SL_{n-1}\subset SL_n$. It is well known that
$\bbk[SL_n/U_n]^{U_{n-1}}$ is a polynomial algebra of Krull dimension $2(n-1)$ and its generators have a simple description,
see e.g. \cite[Sect.\,3]{ap}. The reason is that $SL_n/U_n$ is a spherical
$SL_{n-1}$-variety and the branching rule $SL_n\downarrow SL_{n-1}$ is rather simple.
That is, $\bbk[SL_n/U_n]^{U_{n-1}}$ and $\bbk[SL_n/U_n]^{U_{n}'}$ are polynomial rings of the
same dimension, and also
$\dim U_{n-1}=\dim U'_n$. However, the subgroups $U'_n,\,U_{n-1}\subset SL_n$
are essentially different unless $n=2,\,3$.
\end{ex}

\section{Some properties of algebras of $U'$-invariants} 
\label{sect:gen_prop}

\noindent
The main result of Section~\ref{sect:inv} says that $\bbk[G/U]^{U'}$ is a polynomial
algebra of Krull dimension $2r$. This can also be understood in the other way around,
since $\bbk[G/U]^{U'}$ and $\bbk[G/U']^{U}$ are canonically isomorphic. 
Indeed, for any closed subgroup $H\subset G$, we regard $\bbk[G/H]$  as
subalgebra of $\bbk[G]$:
\[
\bbk[G/H]=\{f \in \bbk[G] \mid f(gh)=f(g) \textrm{ for any } g\in G, h\in H\}.
\]
Any  subgroup of $G$ acts on $G/H$ by left translations. Therefore
\begin{gather*}
\bbk[G/U]^{U'} \simeq \{ f\in \bbk[G] \mid  f(u_1g u_2)=f(g)  \textrm{ for any } g\in G, u_1\in U', u_2\in U\}, 
\\
\bbk[G/U']^{U} \simeq \{ f\in \bbk[G] \mid  f(u_2g u_1)=f(g)  \textrm{ for any } g\in G, u_1\in U', u_2\in U\}.
\end{gather*}
The involutory mapping $(f\in \bbk[G])\mapsto \hat f$, where $\hat f(g)=f(g^{-1})$, takes $\bbk[G/U]^{U'}$ to 
$\bbk[G/U']^{U}$, and vice versa.

One can deduce some properties of $\bbk[G/U']$ using the known structure of
$\bbk[G/U']^{U}$. Set $\ca=\bbk[G/U']$. It is a rational $G$-algebra, which can be decomposed as $G$-module:
\[
       \ca=\bigoplus_{\lb\in \mathfrak X_+} m_{\lb,\ca}\sfr(\lb) .
\]
By Frobenius reciprocity, the multiplicity $m_{\lb,\ca}$ is equal to
$\dim \sfr(\lb^*)^{U'}$. Therefore, it is finite. In our situation,
\[
   \dim \sfr(\lb^*)^{U'}=\dim \sfr(\lb)^{U'}=\prod_{i=1}^r((\lb,\ap_i^\vee)+1) .
\]
In particular, $m_{\varpi_i,\ca}=2$ for any $i$.  One can also argue as follows.

\noindent
The group $G\times G$ acts on $G$ by left and right translations and
the decomposition of $\bbk[G]$ as $G\times G$-module is of the form:
\[
     \bbk[G]= \bigoplus_{\lb\in\mathfrak X_+} \sfr(\lb^*)\otimes \sfr(\lb) \ ,
\]
where the first (resp. second) copy of $G$ in $G\times G$ acts on the first 
(resp. second) factor of tensor product in each summand
\cite[Ch.\,2, \S\,3, Theorem\,3]{kraft}.
Then
\begin{gather}   \label{eq:ca}
    \ca=\bbk[G/U']= \bigoplus_{\lb\in\mathfrak X_+} \sfr(\lb^*)\otimes \sfr(\lb)^{U'},
\\
\label{eq:ca-u}
    \ca^U= \bigoplus_{\lb\in\mathfrak X_+} \sfr(\lb^*)^U\otimes \sfr(\lb)^{U'} .
\end{gather}
In this context, Theorem~\ref{thm:main2} asserts that any basis of the $2r$-dimensional 
vector space \ $
\bigoplus_{i=1}^r\, \sfr(\varpi_i^*)^U\otimes \sfr(\varpi_i)^{U'}$ 
freely generates the polynomial algebra $\ca^U$. It is known that $\bbk[G/U']$ is finitely generated (see \cite[Theorem\,7]{gross86}). Below, we obtain a more precise assertion.

\begin{lm}    \label{lm:fund-cop}
$\ca$ is generated by the copies of fundamental $G$-modules, i.e.,
by the subspace $
\bigoplus_{i=1}^r\, \sfr(\varpi_i^*)\otimes \sfr(\varpi_i)^{U'}$.
\end{lm}\begin{proof}
We know that $\ca^U=\bbk[G/U']^U$ is a polynomial algebra, generated by $2r$ functions.
Using Equations~\eqref{eq:ca} and \eqref{eq:ca-u}, 
one sees that the generators of $\ca^U$ are just the highest vectors of all fundamental
$G$-module sitting in $\ca$.
It follows that the subalgebra of $\ca$ generated by all fundamental $G$-modules is
$G$-stable and contains  the highest vectors of all simple $G$-modules inside $\ca$. Hence
it is equal to $\ca$.
\end{proof}

For a quasi-affine $G/H$,  it is known that  $\bbk[G/H]$ is finitely generated
if and only if there is a  $G$-equivariant embedding $i: G/H \to  V$, where
$V$ is a finite-dimensional $G$-module, such that the boundary of $i(G/H)$ is of
codimension $\ge 2$ \cite[\S\,4]{gr}. 
As $U'$ is unipotent, $G/U'$ is quasi-affine. Hence such an embedding of $G/U'$ exists
and, making use of Lemma~\ref{lm:fund-cop}, we explicitly construct it.

Recall that $f_i$ and $\tilde f_i$ are nonzero weight vectors in $\sfr(\varpi_i)_{\varpi_i}$
and $\sfr(\varpi_i)_{\varpi_i-\ap_i}$, respectively. 

\begin{thm}    \label{thm:model}
Let $p=(f_1,\tilde f_1,\dots,f_r,\tilde f_r)\in 2\sfr(\varpi_1)\oplus\ldots\oplus 2\sfr(\varpi_r)$.
Then   
\begin{itemize}
\item[\sf (i)] \    $G_p=U'$; 
\item[\sf (ii)] \   $\bbk[\ov{G{\cdot}p}]=\bbk[G/U']$ and $\ov{G{\cdot}p}\simeq \spe(\ca)$ is normal;
\item[\sf (iii)] \   $\codim(\ov{G{\cdot}p}\setminus {G{\cdot}p})\ge 2$.
\end{itemize}
\end{thm}\begin{proof}
Part (i) is obvious. Then $G{\cdot}p\simeq G/U'$ and hence $\cb:=\bbk[\ov{G{\cdot}p}]$ is a
subalgebra of $\ca$. By the very construction, $m_{\varpi_i,\cb}\ge 2$. (Consider different
non-trivial projections $\ov{G{\cdot}p} \to \sfr(\varpi_i)$ for all $i$.)  Since $m_{\varpi_i,\cb}\le 
m_{\varpi_i,\ca}=2$ and $\ca$ is generated by the fundamental $G$-modules, we must have
$\cb=\ca$. This yields the rest. 
\end{proof}

\noindent
Let $X$ be 
an algebraic variety equipped with a regular action of $G$. Then $X$ is said to be
a $G$-variety. 
The  ``transfer principle'' 
(\cite[Ch.\,1]{br-these}, \cite[\S\,3]{po86}, \cite[\S\,9]{gr}) asserts that 
\[
   \bbk[X]^{H}\simeq (\bbk[X]\otimes \bbk[G/H])^G   
\]
for any affine $G$-variety $X$ and any subgroup $H\subset G$.
In particular, if $\bbk[G/H]$ is finitely generated, then so is $\bbk[X]^H$.
In view of Lemma~\ref{lm:fund-cop}, this applies to $H=U'$, hence $\bbk[X]^{U'}$ is always
finitely-generated.
Moreover, the polynomiality of $\bbk[G/U']^{U}$ implies that
$\bbk[X]^{U'}$  inherits a number of other good properties from $\bbk[X]$.
Recall that  $\spe(\bbk[X]^{U'})$ is denoted by $X\md U'$; hence $\bbk[X\md U']$ and
$\bbk[X]^{U'}$ are the same objects. 

We often use below the notion of a {\it variety with  rational singularities}.
Let us provide some relevant information for the affine case.
 
a) If  $\phi:\tilde X \to X$ is a resolution of singularities, then $X$ is said to have 
rational singularities if  $H^{0}(\tilde X,\co_{\tilde X})=\bbk[X]$ and
$H^{i}(\tilde X,\co_{\tilde X})=0$ for  $i\ge 1$. In particular, $X$ is necessarily normal.

b) If $X$ has only rational singularities and $G$ is a reductive group acting on $X$, then
$X\md G$ has only rational singularities (Boutot \cite{boutot}).

c) If $X$ has only rational singularities, then $X$ is  Cohen-Macaulay
(Kempf \cite{ke73}). It follows that
 if $X$ is factorial and has rational singularities, then $X$ is Gorenstein. 

\begin{thm}    \label{thm:RS}
Let $X$ be an  irreducible affine $G$-variety. If $X$ has only rational 
singularities, then so has $X\md U'$. Furthermore, if $X$ is factorial, then
$X\md U'$ is factorial, too.
\end{thm}
\begin{proof}
This is a straightforward consequence of known technique. 
Since $\bbk[G/U']^U$ is a polynomial algebra, $G\md U'$ has rational singularities
by Kraft's theorem \cite[Theorem\,1.6]{br-these}, \cite{po86}.
By the transfer principle for $H=U'$, we have
$X\md U'\simeq \bigl(X\times (G\md U')\bigr)\md G$.
Applying Boutot's theorem \cite{boutot} to the right-hand side, we conclude that 
$X\md U'$ has rational singularities. The second assertion stems from the fact that $U'$ has
no non-trivial rational characters. 
\end{proof}

We have $\bbk[X]^{U}\subset \bbk[X]^{U'}$, and both algebras are finitely generated.
Assuming that generators of $\bbk[X]^{U}$ are known,
we obtain a finite set  of generators for $\bbk[X]^{U'}$, as follows.

\begin{thm}    \label{thm:base-X-U'}
Suppose that  $f_1,\dots,f_m$ is a set of $T$-homogeneous generators of  
$\bbk[X]^{U}$ and the weight of $f_i$ is $\lb_i$. (That is, there is a $G$-submodule
$\BV_i\subset \bbk[X]$ such that $\BV_i\simeq \sfr(\lb_i)$ and $f\in (\BV_i)^{U}$.)
Then the union of bases of the spaces\/ $(\BV_i)^{U'}$, $i=1,\dots,m$,  generate $\bbk[X]^{U'}$.
In particular, $\bbk[X]^{U'}$ is generated by at most 
$\sum_{i=1}^m\prod_{j=1}^r \bigl((\lb_i,\ap_j^\vee)+1\bigr)$ functions.
\end{thm}
\begin{proof}
Let $\cb$ be the algebra generated by the spaces $(\BV_i)^{U'}$.
Clearly, $\cb$ is $B/U'$-stable and contains
$\bbk[X]^U$. Hence it meets every simple $G$-submodule of $\bbk[X]$.
Therefore, it is sufficient to prove that $\cb$ contains $U/U'$-cyclic vectors of all simple
$G$-submodules.

We argue by induction on the root order `$\curle$' on the set of dominant weights.
Let $c_i\in (\BV_i)^{U'}$ be the unique $U/U'$-cyclic weight vector. By definition, $c_i\in \cb$.
We normalise $f_i$ and $c_i$ such that 
$E(\lb_i)(c_i)=f_i$, where the operator $E(\lb)$, $\lb\in\mathfrak X_+$, is defined by
$E(\lb):=\prod_{i=1}^r e_i^{(\lb,\ap_i^\vee)}$.
Assume that for any simple $G$-module $\BW$ of type $\sfr(\mu)$ occurring in $\bbk[X]$, with 
$\mu \prec \lb$, the cyclic vector of $\BW$ belong to $\cb$.
Consider  an arbitrary simple submodule $\BV\subset \bbk[X]$ of type $\sfr(\lb)$.
Take a polynomial $P$ in $m$ variables such that $f{=}P(f_1,\dots,f_m)$  is a highest vector
of $\BV$.
Without loss of generality, we may assume that every monomial of $P$ is of  weight
$\lb$. We claim that $P(c_1,\dots, c_m)\ne 0$. Indeed, it is easily seen
that $E(\lb)P(c_1,\dots, c_m)=P\bigl(E(\lb_1)(c_1),\dots, E(\lb_m)(c_m)\bigr)=f$.
The last equality does not guarantee us that $P(c_1,\dots, c_m)\in \BV$. However, 
this means that the projection of this element to $\BV$ is well-defined and it must be
a $U/U'$-cyclic vector of $\BV$, say $c$.
More precisely, $P(c_1,\dots, c_m)=c+\tilde c$, where $\tilde c$ belong to a sum of
simple submodules  of types $\sfr(\nu_i)$ with $\nu_i\prec \lb$.
If $P$ is a monomial, then this follows from the uniqueness of the Cartan component
in tensor products. In our case, the Cartan component of the tensor product associated with
every monomial of $P$ is $\sfr(\lb)$, 
which easily yields the general assertion.
By  definition, $P(c_1,\dots, c_m)\in\cb$, and  by the induction 
assumption, $\tilde c\in \cb$. Thus, $c\in \cb$.
\end{proof}

This theorem provides a good upper bound on the number of generators of $\bbk[X]^{U'}$.
However, it is not always the case that a minimal generating system of $\bbk[X]^U$  is
a part of a minimal generating system of $\bbk[X]^{U'}$. (See examples in Section~\ref{sect:classif}.)

Since $U'$ has no rational characters, $\dim X\md U'=\trdeg \bbk(X)^{U'}=
\dim X-\dim U'+ \min_{x\in X}\dim (U')_x$. To compute the last quantity, we use the existence of
a generic stabiliser for $U$-actions on irreducible $G$-varieties~\cite[Thm.\,1.6]{BLV}.

\begin{lm}   \label{prop:min-stab-U'}
Let $U_\star$ be a generic stabiliser for $(U:X)$. Then 
$\min_{x\in X}\dim (U')_x=\dim (U_\star\cap U')$.
\end{lm}
\begin{proof}
Let $\Psi\subset X$ be a dense open subset of generic points, i.e., 
$U_x$ is $U$-conjugate to $U_\star$ for any $x\in \Psi$.
Since $U'$ is a normal subgroup, $U_x\cap U'$ is also $U$-conjugate to $U_\star\cap U'$.
Thus, all $U'$-orbits in $\Psi$ are of dimension $\dim U'- \dim (U_\star\cap U')$.
\end{proof}

\begin{rmk}   \label{rmk:gen-stab-U'}
1) If $X$ is (quasi)affine, then one can choose $U_\star$ in a canonical way.
Let $\mathcal M(X)$ be the monoid of highest weight of all simple $G$-modules occurring 
in $\bbk[X]$. Then $U_\star$ is the product of all root unipotent subgroup $U^\mu$ ($\mu\in\Delta^+$)
such that $(\mu,\mathcal M(X))=0$ \cite[Ch.\,1, \S\,3]{pa99}.
Equivalently, $U_\star$ is generated by the simple root unipotent  subgroups
$U^{\ap_i}$ such that $(\ap_i,\mathcal M(X))=0$. 
It follows that $U_\star\cap U'=(U_\star,U_\star)$.
This also means that if $\mathcal M(X)$ is known, then $\min_{x\in X}\dim (U')_x$ can effectively  be computed. 

2) The group $U_\star$ is a maximal unipotent subgroup of a generic
stabiliser for the diagonal $G$-action on $X\times X^*$ \cite[Theorem\,1.2.2]{pa99}. Here $X^*$ is the so-called {\it dual $G$-variety}. It coincides with the dual $G$-module, if $X$ is a $G$-module.
Using  tables of generic 
stabilisers for representations of $G$, one can again compute  $U_\star$ and
$(U_\star,U_\star)$.
\end{rmk}

\section{Poincar\'e series of multigraded algebras of $U'$-invariants}
\label{sect:poinc}

\noindent
Let $X$ be an irreducible affine $G$-variety. (Eventually,  we impose other 
constraints on $X$.)
Since $T$ normalises $U'$, it acts on $X\md U'$ and the algebra $\bbk[X]^{U'}$ acquires 
a multigrading (by $T$-weights). 
Our objective is to describe some properties of the
corresponding Poincar\'e series. 
Before we stick to considering  $U'$-invariants, let us give a brief outline of notation and results to be used below.

\noindent
Let $\eus R$ be a finitely generated $\BN^m$-graded $\bbk$-algebra such that 
$\bbk[\eus R]_0=0$. Set $X=\spe(\eus R)$.

\begin{itemize}
\item  The Poincare series of $\eus R$ is (the Taylor expansion of) a rational function in
$t_1,\dots,t_m$:
\[
   \cf(\eus R;\un{t})=P(\un{t})/Q(\un{t})
\]
for some  polynomials $P,Q$. 

\item  If $\eus R$ is Cohen-Macaulay, then  $\Omega_\eus R$ (or $\Omega_X$) stands for
the canonical module of $\eus R$;  \ $\Omega_\eus R$ is
naturally $\BZ^m$-graded such that the Poincar\'e series of $\Omega_\eus R$ is
\[
   \cf(\Omega_\eus R; \un{t})=(-1)^{\dim X}\cf(\eus R;\un{t}^{-1}).
\]
\item  If $\eus R$ is Gorenstein, then the rational function $\cf(\eus R;\un{t})$ satisfies
the equality 
\[
   \cf(\eus R;\un{t}^{-1})=(-1)^{\dim X}\un{t}^{q(X)}\cf(\eus R;\un{t}),
\]
for some $q(X)=(q_1(X),\dots, q_m(X))\in \BZ^m$, and 
the degree of a homegeneous generator $\omega_\eus R$
of $\Omega_\eus R$ is
$ \deg(\omega_\eus R)= q(X)$ 
\cite[Theorem\,6.1]{st78}, \cite[1.12]{st96}.
\item  If $X$ has only rational singularities, then $q_i(X)\ge 0$ and 
$q(X)\ne (0,\dots,0)$  \cite[Proposition\,4.3]{br-these}
\item  Let $G$ be a semisimple group acting on $X$ (of course, it is assumed that $G$ preserves
the $\BN^m$-grading of $\eus R$). Then there is a relationship betweem $\Omega_\eus R$
and $\Omega_{\eus R^G}$ \cite{kn86} and hence between $q(X)$ and $q(X\md G)$, see below. 
\end{itemize}

\noindent
We begin with the case of $X=G$, where $G$ is regarded as $G$-variety with respect to
right translations.  That is, we are going to study the graded
structure of $\ca=\bbk[G/U']$.  Since $G$ is simply-connected, it is a 
factorial variety. Therefore,  $\spe(\ca)=G\md U'$ is factorial 
(and has only rational singularities). In particular, $G\md U'$ is Cohen-Macaulay (=\,CM).
There is the direct sum decomposition
\[
     \ca=\bigoplus_{\gamma\in \mathfrak X}\ca_\gamma ,
\]
where $\ca_\gamma=\{ f\in \ca \mid f(gt)=\gamma(t)f(g) \textrm{ for any } g\in G, t\in T\}$.
The weights $\gamma$ such that $\ca_\gamma\ne 0$ form a finitely generated monoid, 
which is denoted by $\Gamma$. Since $\sfr(\lb)^{U'}$ is a multiplicity free $T$-module,
it follows  from Eq.~\eqref{eq:ca} that, for any
$\lb\in\mathfrak X_+$, 
different copies of $\sfr(\lb^*)$ lie in the different weight spaces $\ca_\gamma$.
More precisely, the corresponding set of weights is $I_\lb$ (see Section~\ref{sect:inv}).
In particular, two copies of $\sfr(\varpi_i^*)$ belong to $\ca_{\varpi_i}$ and 
$\ca_{\varpi_i-\ap_i}$.
Therefore, $\Gamma$ is generated by the weights
$\varpi_i, \varpi_i-\ap_i$, $i=1,\dots,r$. 
Note that the group generated by $\Gamma$ coincides with $\mathfrak X$, since $\Gamma$
contains all fundamental weights.

\begin{lm}    \label{lm:halfspace}
If $G$ has no simple factors $SL_2$ or $SL_3$, then $\Gamma\setminus\{0\}$ lies in an 
open half-space of  $\mathfrak X_\BQ$, $\ca_0=\bbk$, and $\dim \ca_\gamma< \infty$ for all
$\gamma\in\Gamma$.
\end{lm}
\begin{proof}
It is shown in the proof of Proposition~\ref{prop:epi} that $(\rho^\vee,\varpi_i-\ap_i)>0$ for all
$i$. Hence the half-space  determined
by $\rho^\vee$ will do. We have $\ca_0=\bbk[G/TU']=\bbk$, since $TU'$ is epimorphic.
This also implies the last claim, because $\ca$ is finitely generated.
\end{proof}

The algebra $\ca$ is $\Gamma$-graded, and we are going to study the 
corresponding Poincar\'e series. Unfortunately, $\Gamma$ is not always a free monoid. 
Therefore we want to embed $\Gamma$ into a free monoid $\BN^r$. This is always possible,
if $\Gamma$ generates a convex cone in $\mathfrak X_\BQ$, see e.g. 
\cite[Corollary\,7.23]{m-s}. For this reason, 
we assume below that $G$ has no simple factors $SL_2$ or $SL_3$,  and choose an embedding $\Gamma\hookrightarrow \BN^r$.
In other words, we  find $v_1,\dots,v_r\in \mathfrak X$ such that
$\mathfrak X=\bigoplus_{i=1}^r \BZ v_1$ and $\Gamma\subset \bigoplus_{i=1}^r \BN v_1$.
Furthermore, one can achieve that $(v_i, \rho^\vee)>0$ for all $i$. Then
$(v_1,\dots,v_r)$ is said to be a $\Gamma$-{\it adapted basis\/} for $\mathfrak X$.
Thus, every $\gamma\in\Gamma$ gains a unique expression of the
form $\gamma=\sum_i k_i(\gamma)v_i$, \ $k_i(\gamma)\in\BN$.

Now, we define the multigraded Poincar\'e series of $\ca$ as the power series
\[
  \cf(\ca; t_1,\dots,t_r)=\cf(\ca; \un{t})=\sum_{\gamma\in\Gamma}(\dim\ca_{\gamma})\un{t}^\gamma  ,
\]
where $\un{t}^\gamma=t_1^{k_1(\gamma)}\dots t_r^{k_r(\gamma)}$.
As is well-known, $\cf(\ca; \un{t})$ is a rational function. Since $\ca$ is a factorial CM domain,
it is Gorenstein. Therefore, there exists $\un{a}=(a_1,\dots, a_r)\in \BZ^r$ such that
\begin{equation}  \label{eq:f-a-sym}
     \cf(\ca; \un{t}^{-1})=(-1)^{\dim G/U'} \un{t}^{\un{a}} \cf(\ca; \un{t}), 
\end{equation}
where $\un{t}^{-1}=(t_1^{-1},\dots,t_r^{-1})$
\cite[\S\,6]{st78}. Moreover,
since $G\md U'$ has only rational singularities, all $a_i$ are actually non-negative, and
$\un{a}\ne (0,\dots,0)$ \cite[Proposition\,4.3]{br-these}.

Set $b(\ca):=\sum_{i=1}^r a_i v_i\in \mathfrak X$. A priori,  this element might
depend on the choice of an embedding $\Gamma \hookrightarrow \BN^r$.
Fortunately, it doesn't. Roughly speaking, this can be explained via properties of the
{\it canonical module\/} $\Omega_\ca$, which is a free $\ca$-module of rank one.
However, even if we accurately accomplish this program, then we still do 
not find the very element $b(\ca)\in\mathfrak X$.
Therefore, we choose another path. Our plan consists of the following steps:
\begin{enumerate}
\item $\ca^U$ is a polynomial algebra and its Poincar\'e series can be written down 
explicitly;
\item Using the formula for this Poincare series, we determine 
$b(\ca^U)\in \mathfrak X$;
\item Using results of \cite[5.4]{pa99}, we prove that $b(\ca)=b(\ca^U)$.
\end{enumerate}

\noindent
The algebra $\ca^U$ is acted upon by $T\times T$. Two copies
of $T$ acts on $\ca^U\subset \bbk[G]$ via left and right translations. 
For the presentation of Eq.~\eqref{eq:ca-u},  the first (resp. second) copy of $T$ acts on 
the  first (resp. second) factor in tensor products.  Then 
\[
    \ca^U=\bigoplus_{\lb\in\mathfrak X_+, \gamma\in\Gamma} \ca^U_{\lb,\gamma},
\]
where $\ca^U_{\lb,\gamma}=\{f\in \ca^U\subset 
\bbk[G] \mid f(tgt')=\lb(t)^{-1}\gamma(t')f(g) \ \text{ for all }t,t'\in T\}$, and we set
\[
   \cf(\ca^U; \un{s},\un{t})=\sum_{\lb,\gamma}(\dim \ca^U_{\lb,\gamma})\un{s}^\lb \un{t}^\gamma .
\]
Here  $\un{s}=(s_1,\dots,s_r)$ and $\un{s}^\lb=s_1^{n_1}\dots s_r^{n_r}$ if $\lb=\sum_i n_i\varpi_i$.

\begin{prop}   \label{prop:poinc-A^U}  
We have
\[
 \cf(\ca^U; \un{s},\un{t})=
 \prod_{i=1}^r \frac{1}{(1-s_{i^*}\un{t}^{\varpi_i})(1-s_{i^*}\un{t}^{\varpi_i-\ap_i})} ,
\]
where $i^*$ is defined by $(\varpi_i)^*=\varpi_{i^*}$.
\end{prop}
\begin{proof}
This follows from the fact that $\ca^U$ is freely generated by the space
$R=\bigoplus_{i=1}^r\, \sfr(\varpi_i^*)^U\otimes \sfr(\varpi_i)^{U'}$, and the 
$(T\times T)$- weights of  a bi-homogeneous basis of $R$ are 
$(\varpi_i^*, \varpi_i),  (\varpi_i^*, \varpi_i-\ap_i)$, $i=1,\dots,r$.
\end{proof}

\noindent
Of course, $\un{t}^{\varpi_i}$ should be understood as 
$t_1^{k_1(\varpi_1)}\dots t_r^{k_r(\varpi_r)}$, and likewise for $\varpi_i-\ap_i$.
Since  $\sum_i (\varpi_i+\varpi_i-\ap_i)=2\rho-|\Pi|=|\Delta^+\setminus \Pi|$, we readily obtain

\begin{cl}   \label{cor:f-a^u}
$\cf(\ca^U; \un{s}^{-1},\un{t}^{-1})=(s_1\dots s_r)^2 \un{t}^{2\rho-|\Pi|}\cf(\ca^U; \un{s},\un{t})$.
\end{cl}

One can disregard (for a while) the $\mathfrak X_+$-grading of $\ca^U$ and consider 
only the $\Gamma$-grading induced from $\ca$. This amount to letting $s_i=1$ for all $i$.  
Then we obtain
$b(\ca^U)=2\rho-|\Pi|$, and, surely, this does not depend on the choice of 
$\Gamma \hookrightarrow \BN^r$. Thus, we have completed steps (1) and (2) of the 
above plan.

Now, we recall a relationship between the multigraded  Poincar\'e series of algebras $\bbk[X]$
and $\bbk[X]^U$. For $G$-modules, these results are due to 
M.~Brion \cite[Ch.\,IV]{br-these}, \cite[Theorem\,2]{br83}. A general version is found in 
\cite{pa95}, \cite[Ch.\,5]{pa99}. We will consider two types of conditions imposed on 
$G$-varieties $X$:

\vskip1ex
\hbox to \textwidth{\ ($\mathfrak C_1$) \hfill $\left\{ \text{
\parbox{410pt}{ $X$ is an irreducible factorial $G$-variety with only rational singularities 
and $\bbk[X]^G=\bbk$.}}\right.$}

\vskip1.5ex
\hbox to \textwidth{\ ($\mathfrak C_2$) \hfill \hfill $\left\{ \text{
\parbox{410pt}{ $X$ is an irreducible factorial $G$-variety with only rational singularities;
$\bbk[X]$ is $\BN^m$-graded, 
$\bbk[X]=\bigoplus_{n\in\BN^m} \bbk[X]_n$, and $\bbk[X]_0=\bbk$. }}\right.$}

\vskip1ex\noindent
In particular, $X$ is Gorenstein in both cases. Suppose $X$ satisfies ($\mathfrak C_2$).
The Poincar\'e series of the Gorenstein algebra $\bbk[X]$ satisfies an equality of the form
\begin{equation}  \label{eq:symmetry-X}
  \cf(\bbk[X]; \un{t}^{-1})=(-1)^{\dim X} \un{t}^{q(X)}\cf(\bbk[X]; \un{t}),
\end{equation}
where $\un{t}=(t_1,\dots,t_m)$ and $q(X)=(q_1(X),\dots,q_m(X))$. 
The affine variety  $X\md U$ 
inherits all good properties of $X$, i.e., it is irreducible, factorial, etc. Furthermore,
$\bbk[X]^U$ is naturally $\mathfrak X_+ \times \BN^m$-graded, and one defines the Poincar\'e
series
\[
  \cf(\bbk[X]^ U; \un{s}, \un{t})=\sum_{\lb\in\mathfrak X_+, {n}\in\BN^m} (\dim \bbk[X]^U_{\lb, n})
  \un{s}^\lb \un{t}^{{n}} .  
\]
Since $X\md U$ is again Gorenstein, this series satisfies
an equality of the form
\[
\cf(\bbk[X]^U; \un{s}^{-1}, \un{t}^{-1})=(-1)^{\dim X\md U} \un{s}^{\un{b}}\, 
\un{t}^{q(X\md U)}  \cf(\bbk[X]^ U; \un{s}, \un{t}) 
\]
for some $\un{b}=\un{b}(X\md U)=(b_1,\dots, b_r)$ and 
$q(X\md U)=(q_1(X\md U),\dots,q_m(X\md U))$. 

\begin{thm}[see \protect {\cite[Theorem\,5.4.26]{pa99}}]    \label{thm:p99}
Suppose that $X$ satisfies condition ($\mathfrak C_2$). Then
\begin{enumerate}
\item $0\le b_{i}\le 2$; 
\item $0\le q_{i}(X\md U)\le q_{i}(X)$ for all $i$;
\item the following conditions are equivalent:
\begin{itemize}
\item  $\un{b}=(2,\dots,2)$;
\item  For $D=\{z\in X\mid  \dim U_z>0\}$, we have $\codim_X D\ge 2$;
\item  ${q}(X\md U)={q}(X)$;
\end{itemize}
\end{enumerate}
\end{thm}

\noindent
Let us apply this theorem to the $G$-variety $\spe(\ca)=G\md U'$.
The algebra $\ca$ is $\Gamma$-graded and hence suitably $\BN^r$-graded, as explained before. Note that $\spe(\ca)$ satisfies both conditions
($\mathfrak C_1$) and ($\mathfrak C_2$). At the moment, we consider $X=\spe(\ca)$ as variety
satisfying condition ($\mathfrak C_2$), with $m=r$.
Comparing Eq.~\eqref{eq:f-a-sym} and \eqref{eq:symmetry-X}, we see that
$\un{a}=q(X)$.
Proposition~\ref{prop:poinc-A^U} and 
Corollary~\ref{cor:f-a^u} show that here $\un{b}(X\md U)=(2,\dots,2)$ and $q(X\md U)$ 
corresponds  to $b(\ca^U)=2\rho-|\Pi|$.  Now, 
Theorem~\ref{thm:p99}(3) guarantee us that  $q(X)=q(X\md U)$, i.e.,
\begin{equation}   \label{eq:q(a)=q(au)}
b(\ca)=b(\ca^U)=2\rho-|\Pi| .
\end{equation}

\noindent
This completes our computation of $b(\ca)$. Note that we computed  $b(\ca)$ without
knowing an explicit formula of the Poincar\'e series $\cf(\ca; \un{t})$.

Our next goal is to obtain analogues of results of \cite[5.4]{pa99}, 
where $U$ is replaced with $U'$, i.e., results on Poincar\'e series of algebras $\bbk[X]^{U'}$.

\noindent
Suppose $X$ satisfies $(\mathfrak C_1)$.
The algebra $\bbk[X]^{U'}$ is  $\Gamma$-graded, and we consider the 
Poincar\'e series
\[
   \cf(\bbk[X]^{U'}; \un{t})=\sum_{\gamma\in\Gamma}
   \dim \bbk[X]_{\gamma}^{U'}\,\un{t}^\gamma ,
\]
where $\bbk[X]_{\gamma}^{U'}=\{f\in \bbk[X]^{U'} \mid f(t.z)=\gamma(t)^{-1}f(z)\}$ and, as 
above, $\un{t}^\gamma$ is determined via the  choice of a $\Gamma$-adapted
basis $(v_1,\dots,v_r)$ for $\mathfrak X$.
The assumption $\bbk[X]^G=\bbk$ and the convexity of the cone generated by $\Gamma$ guarantee us that $\bbk[X]^{U'}_0=\bbk$ and all spaces $\bbk[X]_{\gamma}^{U'}$ are 
finite-dimensional.
Since $X\md U'$ is again factorial, with only rational singularities (Theorem~\ref{thm:RS}), 
it is Gorenstein and hence
\[
   \cf(\bbk[X]^{U'};  \un{t}^{-1})=(-1)^{\dim X\md U'}\un{t}^a\, 
    \cf(\bbk[X]^{U'}; \un{t})
\] 
for some $a=(a_1,\dots,a_r)\in \BN^r$. Using the basis 
$(v_1,\dots,v_r)$, we set $b(X\md U')=\sum_{i=1}^r a_i v_i\in \mathfrak X$. 

\begin{thm}   \label{thm:analog-5-4-21}
Suppose that $X$ satisfies $(\mathfrak C_1)$. Then
\begin{enumerate}
\item $0\le b(X\md U')\le b(\ca)=2\rho-|\Pi|$ \\(componentwise, with respect to 
any  $\Gamma$-adapted basis $v_1,\dots,v_r$); 
\item the following conditions are equivalent:
\begin{itemize}
\item[a)] \   $b(X\md U')= 2\rho-|\Pi|$;
\item[b)] \   For $D=\{x\in X\mid  \dim (U')_x>0\}$, we have $\codim_X D\ge 2$;
\end{itemize}
\end{enumerate}
\end{thm}
\begin{proof}
Using our results on $\ca$ and $\ca^U$ obtained above, one can easily adapt the
proof of  \cite[Theorem\,5.4.21]{pa99}. For the reader's convenience, we recall the argument. 

(1) We have $0\le b(X\md U')$, since $X\md U'$ has rational singularities.\\
Set $Z=X\times (G\md U')$. It is a factorial $G$-variety with only rational singularities
and $\bbk[Z]=\bbk[X]\otimes \ca$. Define the $\Gamma$-grading of $\bbk[Z]$ by
$\bbk[Z]_\beta=\bbk[X]\otimes \ca_\beta$, $\beta\in\Gamma$. By the transfer principle,
$\bbk[Z]^G \simeq \bbk[X]^{U'}$ and the $\Gamma$-grading of 
$\bbk[X]^{U'}$
corresponds under this isomorphism to the $\Gamma$-grading of
$\bbk[Z]^G$ as subalgebra of $\bbk[Z]$.

In this situation (a semisimple group $G$ acting on a factorial variety $Z$ with only 
rational singularities), one can apply results of Knop to
the quotient morphism $\pi_G: Z\to Z\md G$. Set $m=\max_{z\in Z} \dim G.z$. 
Recall that $\Omega_X$ is the canonical module of  $\bbk[X]$.
By  Theorems~1,2 in \cite{kn86}, there is an injective $G$-equivariant homomorphism of degree 0
of graded $\bbk[Z]$-modules
\[
  \bar\gamma: \Omega_Z   
  \to\wedge^m \g^*\otimes \pi_G^* (\Omega_{Z\md G}).
\]
Here $\Omega_Z=\Omega_X\otimes \Omega_{G\md U'}$ and grading of $\Omega_Z$
comes
from the grading of $\Omega_{G\md U'}$.
The injectivity of $\bar\gamma$ implies that 
\[
  b(X\md U')=
  \left\{ \text{\parbox{5.1cm}{degree of a homogeneous generator of  $\Omega_{X\md U'}
   \simeq\Omega_{Z\md G}$}}
   \right\} \le
   \left\{ \text{\parbox{4.8cm}{degree of a homogeneous generator of $\Omega_{G\md U'}
   $}}
   \right\}=b(\ca). 
 \]
This yields the rest of part (1). 

(2) To prove the equivalence of  a) and b), we replace each of them with an equivalent
condition stated in terms of $Z$:

a')  $\deg (\omega_{Z\md G})=\deg(\omega_Z)$;

b')  $\codim_Z \tilde D\ge 2$, where $\tilde D=\{z\in Z \mid \dim G_z >0\}$.

\noindent The argument in part (1) shows that  a) and a') are equivalent.
The equivalence of b) and b') follows from the fact that $G/U'$ is dense in $G\md U'$
and the complement is of codimension $\ge 2$, see Theorem~\ref{thm:model}.

The injectivity  and $G$-equivariance of $\bar\gamma$ means that there is
$c\in (\wedge^m \g^*\otimes \bbk[Z])^G$ such  that 
$\bar\gamma(\omega_Z)=c{\cdot}\omega_{Z\md G}$.
We can regard $c$ as $G$-equivariant morphism $c': Z\to \wedge^m \g^*$.
It is shown in \cite{kn86} that if $\dim G.z=m$ and $z\in Z_{reg}$, then 
$c'(z)$ is nonzero and it yields (normalised) Pl\"ucker coordinates of the $m$-dimensional space
$\g_z^\perp \subset \g^*$.

Assume a'), i.e.,  $\deg (\omega_{Z\md G})=\deg(\omega_Z)$. Then $\deg c=0$, i.e., 
\[
   c\in (\wedge^m \g^*\otimes \bbk[Z]_0)^G=(\wedge^m \g^*\otimes \bbk[X])^G.
\]
This means that $c'$ can be pushed through the projection to $X$:
\[
    Z=X\times (G\md U') \to X\to \wedge^m \g^*.
\]
Let $z=(x,v)\in X\times (G\md U') $ be a generic point, i.e.,  
$x\in X_{reg}$, $v\in G/U'$, and  $\dim G.z=m$.
Since $c'(z)$ depends only on $x$, we see that $\g_z$ does not depend on $v$. But this is only possible if $\dim\g_z=0$, that is, $m=\dim G$.  This already proves that $\codim_Z\tilde D \ge 1$.
If $\codim_Z\tilde D = 1$, then formulae (6), (7), (12) in \cite{kn86} show that
$\tilde D=\{z\in Z\mid c'(z)=0\}$. However,  $\wedge^m \g^*$ is the 
trivial 1-dimensional $G$-module, hence $c\in \bbk[X]^G=\bbk$. That is, $c'$ is a constant (nonzero)
mapping. This contradiction shows that $\codim_Z\tilde D \ge 2$.

Conversely, if b') holds, then $\tilde D$ is a proper subvariety of  $Z$, i.e., $m=\dim G$ and 
$c\in (\wedge^{\dim G}\g^*\otimes \bbk[Z])^G=\bbk[Z]^G$. Furthermore, since
$\codim_Z\tilde D \ge 2$,  $c$ has no zeros on $Z$ (because $Z$ is normal and
$c'(z)=c(z)\ne 0$ for any
$z\in Z_{reg}\setminus \tilde D$.) It follows that $c$ is constant, $\deg c=0$ and hence
$\deg (\omega_{Z\md G})=\deg(\omega_Z)$.
\end{proof}

\noindent
If $X$ satisfies $(\mathfrak C_2)$, then
the algebra $\bbk[X]^{U'}$ is naturally $\Gamma\times \BN^m$-graded, and we consider the 
Poincar\'e series
\[
   \cf(\bbk[X]^{U'}; \un{s}, \un{t})=\sum_{n\in\BN^m, \gamma\in\Gamma}
   \dim \bbk[X]_{n,\gamma}^{U'}\,\un{s}^n \un{t}^\gamma ,
\]
where $\bbk[X]_{n,\gamma}^{U'}=\{f\in \bbk[X]_{n}^{U'} \mid f(t.z)=\gamma(t)^{-1}f(z)\}$ and 
$\un{t}^\gamma$ is as above.
Since $X\md U'$ is again Gorenstein, we have
\[
   \cf(\bbk[X]^{U'}; \un{s}^{-1}, \un{t}^{-1})=(-1)^{\dim X\md U'}\un{t}^a\, \un{s}^{q(X\md U')}
    \cf(\bbk[X]^{U'}; \un{s}, \un{t})
\] 
for some $a=(a_1,\dots,a_r)\in \BN^r$ and $q(X\md U')\in \BN^m$. Using the basis 
$(v_1,\dots,v_r)$, we set $b(X\md U')=\sum_{i=1}^r a_i v_i\in \mathfrak X$. 
The following is a $U'$-analogue of Theorem~\ref{thm:p99}.

\begin{thm}    \label{thm:analog-5-4-26}
Suppose that $X$ satisfies $(\mathfrak C_2)$. Then
\begin{enumerate}
\item $0\le b(X\md U')\le b(\ca)=2\rho-|\Pi|$ \\(componentwise, with respect to 
any  $\Gamma$-adapted basis $v_1,\dots,v_r$); 
\item $0\le q_i(X\md U')\le q_i(X)$ for all $i$;
\item the following conditions are equivalent:
\begin{itemize}
\item[(i)] \  $b(X\md U')= 2\rho-|\Pi|$;
\item[(ii)] \   For $D=\{x\in X\mid  \dim (U')_x>0\}$, we have $\codim_X D\ge 2$;
\item[(iii)] \   ${q}(X\md U')={q}(X)$.
\end{itemize}
\end{enumerate}
\end{thm}

\noindent
We leave it to the reader to adapt the proof of Theorem~5.4.26 in \cite{pa99}
to the  $U'$-setting.

These results may (and will) be applied to describing $G$-varieties $X$ with
polynomial algebras $\bbk[X]^{U'}$.
Suppose for simplicity that  $\bbk[X]$ is $\BN$-graded (i.e., $m=1$). If $f_1,\dots,f_s$ are algebraically independent homogeneous generators of 
$\bbk[X]^{U'}$, then $\sum\deg f_i =q(X\md U')\le  q(X)$. In particular, if $X$ is a $G$-module
with the usual $\BN$-grading of $\bbk[X]$, then $\sum\deg f_i \le \dim X$.
Similarly, if $\omega_i$ is the $T$-weight of $f_i$, then $\sum_{i=1}^s\omega_i \le
2\rho-|\Pi|$. 
The idea to use an {\sl a priori} information on the Poincar\'e series for classifying
group actions with polynomial algebras of invariants is not new. It goes back to 
T.A.\,Springer \cite{sp80}. Since then it was applied many times to various group actions.

\section{Some combinatorics related to $U'$-invariants}
\label{sect:combin}

\noindent
In previous sections, we have encountered some interesting objects in $\mathfrak X$ 
related to the study of $U'$-invariants. These are $b(\ca)=2\rho-|\Pi|$, the set of $T$-weights
in $\sfr(\lb)^{U'}$ (denoted $I_\lb$), and the monoid $\Gamma$ generated by 
$\varpi_i, \varpi_i-\ap_i$ for all $i\in \{1,\dots,r\}=:[r]$.

\begin{prop}   \label{lm:dom_weights}
\ \phantom{x} 
\begin{itemize}
\item[(i)] \ If $G$ has no simple ideals $SL_2$, then $2\rho-|\Pi|$ is a strictly dominant weight;
\item[(ii)] \  For any $\lb\in\mathfrak X_+$, the weight $|I_\lb|$ is dominant.
Furthermore, $(|I_\lb|, \ap_i)>0$ if and only if there is $j$ such that $(\lb,\ap_j^\vee)>0$
and $(\ap_i,\ap_j)>0$.
\end{itemize}
\end{prop}
\begin{proof}
(i) is obvious.

\noindent
(ii) Recall that $I_\lb=\{\lb -\sum_{i=1}^r c_i\ap_i\mid 0\le c_i\le (\lb,\ap_i^\vee), \ 
 i=1,\dots,r\}$.
Choose $i\in [r]$ and slice $I_\lb$ into the layers, where all coordinates $c_j$ with $j\ne i$ are
fixed, i.e., consider
\[
   I_\lb(c_1,\dots,\widehat{c_i},\dots,c_r)=\{\lb -\sum_{j:j\ne i} c_j\ap_j- c_i\ap_i \mid 0\le c_i\le (\lb,\ap_i^\vee)\}.
\]
Then one easily verifies that
$\bigl(|I_\lb(c_1,\dots,\widehat{c_i},\dots,c_r)|, \ap_i^\vee\bigr)=((\lb,\ap_i^\vee)+1)(-\sum_{j\ne i}c_j\ap_j,\ap_i^\vee)\ge 0$. Hence $(|I_\lb|,\ap_i^\vee)\ge 0$, and the condition of positivity is also
inferred.
\end{proof}

Let $\mathcal C$ be the cone in $\mathfrak X_\BQ$ generated $\Gamma$, i.e., 
by all weights $\varpi_i, \varpi_i-\ap_i$. 
Consider the dual cone 
$\check{\mathcal C}:=\{\eta\in \mathfrak X_\BQ \mid (\eta,\varpi_i)\ge 0 \ \& \  
(\eta,\varpi_i-\ap_i)\ge 0  \text{ for all } \ i\}$.

\begin{thm}   \label{thm:dual-cone}
The cone $\check{\mathcal C}$ is generated by the non-simple positive roots.
\end{thm}
\begin{proof}
1) \ Let $\mathcal K$ denote the cone in $\mathfrak X_\BQ$ generated by 
$\Delta^+\setminus \Pi$. 
It is easily seen that $\mathcal K\subset \check{\mathcal C}$. Indeed, let
$\delta\in \Delta^+\setminus \Pi$. Then $(\varpi_i,\delta)\ge 0$.
If $s_i\in W$ is the
reflection  corresponding to $\ap_i\in\Pi$, then $s_i(\varpi_i)=\varpi_i-\ap_i$
and $s_i(\delta)\in \Delta^+$.
Hence $(\varpi_i-\ap_i,\delta)=(\varpi_i,s_i(\delta))\ge 0$.

2) \ Conversely, we prove that $\check{\mathcal K}\subset \mathcal C$.
We construct a partition of $\check{\mathcal K}$ into finitely many simplicial cones, and show
that each cone belong in $\mathcal C$.

Suppose that $\mu\in\mathfrak X$ and $(\mu,\delta)\ge 0$ for all $\delta\in \Delta^+\setminus
\Pi$. 
Set $J=J_{(\mu)}=\{ j\in[r]\mid (\mu,\ap_j)<0\}$. 
We identify the elements of $[r]$ with the corresponding
nodes of the Dynkin diagram of $G$.
The obvious but crucial observation is that the nodes in $J$ are disjoint on the 
Dynkin diagram. (Such subsets $J$ are said to be {\it disjoint}.)

{\bf Claim.} {\it The $r$ vectors $\varpi_i$ ($i\not\in J$), $\varpi_j-\ap_j$ ($j\in J$) form a basis
for  $\mathfrak X_\BQ$}.
\\
{\sl Proof.} \ Since $J$ is disjoint, $\prod_{j\in J}s_j\in W$ takes these $r$ vectors to 
$\varpi_1,\dots, \varpi_r$.

Thus, we can uniquely write 
\[
  \mu=\sum_{i\not\in J}b_i\varpi_i +\sum_{j\in J}a_j(\varpi_j-\ap_j), \quad b_i,a_j\in \BQ .
\]
By the assumption, $(\mu,\ap_i)\ge 0$ if and only if $i\not\in J$.
For $j\in J$, we have $(\mu,\ap_j^\vee)=-a_j<0$, i.e., $a_j>0$.
It is therefore suffices to prove that all $b_i$ are nonnegative.
Choose any $i\not\in J$. Let $J[i]$ denote the set of all nodes in $J$ that are adjacent to
$i$. Set $w_i=\prod_{j\in J[i]}s_j \in W$. (If $J[i]=\varnothing$, then $w=1$.)
Then $w_i(\ap_i)$ is either $\ap_i$ or a non-simple positive root.
In both cases, we know that $(\mu, w_i(\ap_i))\ge 0$. On the other hand, this scalar product is equal to
$(w_i(\mu),\ap_i)=b_i(\varpi_i,\ap_i)$. Thus, each $b_i$ is nonnegative and $\mu\in 
\mathcal C$.
\end{proof}

\begin{rmk}
The argument in the second part of proof shows that $\mathcal C$ is the union of simplicial cones
parametrised  by the disjoint subset of the Dynkin diagram. For any such set $J\subset [r]$, let
$\mathcal C_J$ denote the simplicial cone generated by 
$\varpi_i$ ($i\not\in J$), $\varpi_j-\ap_j$ ($j\in J$). Then
\[
    \mathcal C= \bigcup_{J \text{ disjoint}}  \mathcal C_J .
\]
Here $\mathcal C_\varnothing$ is the dominant Weyl chamber and 
$\mathcal C_J=(\prod_{j\in J} s_j)\mathcal C_\varnothing$. Furthermore, if 
$\mathcal C_J^o=
\{\sum_{i\not\in J}b_i\varpi_i +\sum_{j\in J}a_j(\varpi_j-\ap_j)\mid a_j >0,\ b_i\ge 0\}$, then
\[
    \mathcal C= \bigsqcup_{J \text{ disjoint}}  \mathcal C_J^o .
\]
\end{rmk}

\begin{rmk}
It is a natural problem to determine the edges (one-dimensional faces) of the cone 
$\check{\mathcal C}$. We can prove that, for $\GR{A}{r}$ and $\GR{C}{r}$, the edges are 
precisely the roots of height 2 and 3. 
However, this is no longer true in the other cases, because 
a root of height 4 is needed.
\end{rmk}

\section{Irreducible representations of simple Lie algebras with 
polynomial algebras of $U'$-invariants}
\label{sect:classif}

In this section, we obtain the list of all irreducible representations of simple Lie algebras with 
polynomial algebras of $U'$-invariants. If $G=SL_2$, then $U'$ is trivial and so is the classification problem. Therefore we assume that $\rk G\ge 2$.

\begin{thm}   \label{thm:classif}
Let $G$ be a connected simple algebraic  group with $\rk G\ge 2$ and 
$\sfr(\lb)$ a simple $G$-module. The following conditions are equivalent:
\begin{itemize}
\item[\sf (i)] \  $\bbk[\sfr(\lb)]^{U'}$ is generated by homogeneous algebraically independent
polynomials;
\item[\sf (ii)] \  Up to the symmetry of the Dynkin diagram of $G$, the weight  $\lb$ occurs 
in Table~\ref{tab1}. 
\end{itemize}
\end{thm}
\noindent
For each item in the table, the degrees and
weights of homogeneous algebraically independent generators are indicated. 
We use the numbering of simple roots as in \cite{t41}.

\begin{longtable}{l>{$}c<{$}c}
\caption{The simple $G$-modules with polynomial algebras of $U'$-invariants}  \label{tab1}\\
$G$ & \lb & Degrees and weights of homogeneous generators of $\bbk[\sfr(\lb)]^{U'}$\endfirsthead  
 $G$ & \lb & Degrees and weights of homogeneous generators of $\bbk[\sfr(\lb)]^{U'}$
\endhead
  \hline
$\GR{A}{r}$  {\small $(r{\ge} 2)$}  & \varpi_r &  $(1,\varpi_1), (1,\varpi_1-\ap_1)$ \\  \hline
$\GR{A}{2r-1}$  & \varpi_{2}^* & $(1,\varpi_2), (2,\varpi_4),\dots, (r-1,\varpi_{2r-2}),(r,\un{0}),$
 \\     {\small $(r{\ge} 2)$}   
 &   &   $(1,\varpi_2-\ap_2), (2,\varpi_4-\ap_4),\dots, (r-1,\varpi_{2r-2}-\ap_{2r-2})$ \\   
\hline
$\GR{A}{2r}$  & \varpi_{2}^* & $(1,\varpi_2), (2,\varpi_4),\dots (r-1,\varpi_{2r-2}),(r,\varpi_{2r}),$ \\
  {\small $(r{\ge} 2)$}      &   &   $(1,\varpi_2-\ap_2), (2,\varpi_4-\ap_4),\dots, (r,\varpi_{2r}-\ap_{2r})$ \\   \hline
$\GR{B}{r}$    & \varpi_1 &  $(1,\varpi_1), (1,\varpi_1-\ap_1), (2,\un{0})$ \\  \hline
$\GR{B}{3}$    & \varpi_3 &  $(1,\varpi_3), (1,\varpi_3-\ap_3), (2,\un{0})$ \\  \hline
$\GR{B}{4}$    & \varpi_4 &  $(1,\varpi_4), (1,\varpi_4-\ap_4), 
(2,\varpi_1), (2,\varpi_1-\ap_1),(2,\un{0})$ \\  
\hline
$\GR{B}{5}$    & \varpi_5 &  $(1,\varpi_5), (1,\varpi_5{-}\ap_5),  
(2,\varpi_1), (2,\varpi_1{-}\ap_1), (2,\varpi_2), (2,\varpi_2{-}\ap_2), $ 
\\
  & &  $ (3,\varpi_5), 
(3,\varpi_5{-}\ap_5), (4,\varpi_3{-}\ap_3), (4,\varpi_4), (4,\varpi_4{-}\ap_4), (4,\un{0})$ \\  
\hline
$\GR{C}{r}$    & \varpi_1 &  $(1,\varpi_1), (1,\varpi_1-\ap_1)$ \\  
\hline
$\GR{D}{r}$ {\small $(r{\ge} 4)$}    & \varpi_1 &  $(1,\varpi_1), (1,\varpi_1-\ap_1), (2,\un{0})$ \\  
\hline
$\GR{D}{5}$    & \varpi_5 &  $(1,\varpi_4), (1,\varpi_4-\ap_4), (2,\varpi_1), (2,\varpi_1-\ap_1)$\\
\hline
$\GR{D}{6}$    & \varpi_6 &  
$(1,\varpi_6), (1,\varpi_6{-}\ap_6), (2,\varpi_2),  (2,\varpi_2{-}\ap_2), $  \\
 & & 
$ (3,\varpi_6), (3,\varpi_6{-}\ap_6), (4,\varpi_4{-}\ap_4), (4,\un{0})$ \\
\hline
$\GR{E}{6}$    & \varpi_1 &  $(1,\varpi_5), (1,\varpi_5-\ap_5), 
(2,\varpi_1), (2,\varpi_1-\ap_1),(3,\un{0})$ \\
\hline 
$\GR{E}{7}$    & \varpi_1 &  $(1,\varpi_1), (1,\varpi_1{-}\ap_1), (2,\varpi_6), 
(2,\varpi_6{-}\ap_6),$
\\
& & $(3,\varpi_1), (3,\varpi_1{-}\ap_1), (4,\varpi_2{-}\ap_2), (4,\un{0}) $  \\ 
\hline
$\GR{F}{4}$    & \varpi_1 & $(1,\varpi_1), (1,\varpi_1{-}\ap_1), (2,\varpi_1), 
(2,\varpi_1{-}\ap_1), (3,\varpi_2{-}\ap_2), (2,\un{0}), (3,\un{0})$ \\
\hline 
$\GR{G}{2}$    & \varpi_1 &  $(1,\varpi_1), (1,\varpi_1-\ap_1), (2,\un{0})$ \\  \hline                      
\end{longtable}

\vskip1ex\noindent
Before starting the proof, we develop some more tools.
Let $V$ be a simple $G$-module.
{\sl A posteriori}, it appears to be true that if $\rk G>1$ and $\bbk[V]^{U'}$ is polynomial,
then so is $\bbk[V]^U$. Therefore our list is contained in Brion's list of representations with polynomial algebras $\bbk[V]^U$ \cite[p.\,13]{br83}.
However, we could not find a conceptual proof. The following is a reasonable substitute:

\begin{prop}     \label{kvg-free}
Suppose that $\bbk[V]^{U'}$ is polynomial and $G\ne SL_3$.
Then $\bbk[V]^G$ is polynomial. 
\end{prop}
\begin{proof} As in Section~\ref{sect:poinc},
consider the $\Gamma$-grading   \ 
  $\displaystyle \bbk[V]^{U'}=\bigoplus_{\gamma\in\Gamma}\bbk[V]^{U'}_{\gamma}$.
\\
If $G\ne SL_3$, then 
$TU'$ is epimorphic and hence  $\bbk[V]^{U'}_{0}= \bbk[V]^G$.
Furthermore, since $\Gamma$ generates a convex cone,  
$\bigoplus_{\gamma\ne 0}\bbk[V]^{U'}_{\gamma}$ is a complementary ideal to 
$\bbk[V]^G$.
In this situation, a minimal system of homogeneous generators for $\bbk[V]^G$
is a part of a minimal system of homogeneous generators for $\bbk[V]^{U'}$.
\end{proof}

\begin{rmk}   \label{rmk-SL3}
For $G= SL_3$, it is not hard to verify that the only representations with polynomial algebras of
$U'$-invariants are $\sfr(\varpi_1)$ and $\sfr(\varpi_2)$. The reason is that
$U'$ is the maximal unipotent subgroup of $SL_2\subset SL_3$. 
Therefore, by classical Roberts' theorem, we have
$\bbk[V]^{U'}\simeq \bbk[ V\oplus \sfr_1]^{SL_2}$, where 
$V$ is regarded as $SL_2$-module and $\sfr_1$ is the tautological $SL_2$-module.
All  $SL_2$-modules with polynomial algebras of invariants are known \cite[Theorem\,4]{po83}, and the restriction
of the simple $SL_3$-modules to $SL_2$ are easily computed.
\end{rmk}

\noindent
Let $U'_\star$ denote a $U'$-stabiliser of minimal dimension for points in $\sfr(\lb)$.
Recall that Lemma~\ref{prop:min-stab-U'}  and Remark~\ref{rmk:gen-stab-U'} provide 
effective tools for computing $U'_\star$ and $\dim U'_\star$. If a ring of invariants $\mathfrak A$ 
is polynomial, then elements of a minimal generating system of $\mathfrak A$ are said to be 
{\it basic invariants}.

\begin{prop}   \label{prop:upper-bound}
Suppose that\/ $\bbk[\sfr(\lb)]^{U'}$ is polynomial and $G\ne SL_3$. Then 
\[
  \dim \sfr(\lb)\le 2\dim(U'/U'_\star) + \prod_{i=1}^r((\lb,\ap_i^\vee)+1).
\]
In particular,  $\dim \sfr(\lb)\le 2\dim U' + \prod_{i=1}^r((\lb,\ap_i^\vee)+1)$.
\end{prop}
\begin{proof}
We consider $\bbk[\sfr(\lb)]$ with the usual $\BN$-grading by the total degree of polynomial.
Then $\bbk[\sfr(\lb)]^{U'}$ is $\Gamma\times \BN$-graded, and it has a minimal generating
system that consists of (multi)homogeneous polynomials.
Let $f_1,\dots,f_s$ be such a system.
By Theorem~\ref{thm:analog-5-4-26}(ii), we have 
\[ 
   \sum \deg(f_i)=q(\sfr(\lb)\md U')\le q(\sfr(\lb))=\dim \sfr(\lb) .
\]
On the other hand, 
$s=\dim \sfr(\lb)-\dim (U'/U'_\star)$ and the number of basic invariants of degee 1 equals 
$a(\lb):=\prod_{i=1}^r((\lb,\ap_i^\vee)+1)$. 
All other basic invariants are of degree $\ge 2$, and we obtain
\[
   a(\lb)+2(\dim \sfr(\lb)-\dim (U'/U'_{\star})-a(\lb))=a(\lb)+2(s-a(\lb))
   \le q(\sfr(\lb)\md U')\le \dim \sfr(\lb).
\]
Hence $\dim \sfr(\lb)\le 2\dim(U'/U'_\star) + a(\lb)$.
\end{proof}

\un{\it Proof of Theorem~\ref{thm:classif}}.\\
$\mathsf {(i)\Rightarrow (ii)}$.  The list of irreducible representations of simple Lie algebras 
with polynomial algebras $\bbk[V]^G$ is
obtained in \cite{KPV}. By Proposition~\ref{kvg-free}, it suffices 
to prove that the representations  
in \cite[Theorem\,1]{KPV} that do not appear in Table~\ref{tab1} cannot have a polynomial 
algebra of $U'$-invariants. 
The list of representation in question is the following:

\begin{itemize}
\item[I)]  \ $(\GR{A}{r},\varpi_3)$, $r=6,7,8$; $(\GR{A}{7},\varpi_4)$; $(\GR{A}{2},3\varpi_1)$;
$(\GR{B}{r},2\varpi_1)$, $r\ge 2$; $(\GR{D}{r},2\varpi_1)$, $r\ge 4$;
$(\GR{B}{6},\varpi_6)$; $(\GR{D}{8},\varpi_8)$; 
$(\GR{C}{r},\varpi_2)$, $r\ge 4$; $(\GR{C}{4},\varpi_4)$;
the adjoint representations.

\item[II)] \ $(\GR{A}{5},\varpi_3)$; $(\GR{C}{3},\varpi_2)$; $(\GR{C}{3},\varpi_3)$;
$(\GR{D}{7},\varpi_7)$; $(\GR{A}{r},2\varpi_r)$.
\end{itemize}

\textbullet \ \ For list I), a direct application of 
Proposition~\ref{prop:upper-bound} yields the conclusion. 
For instance, consider $\sfr(\varpi_3)$  for $\GR{A}{r}$ and $r=6,7,8$.
Here $a(\varpi_3)=2$ and the second inequality in
Proposition~\ref{prop:upper-bound} becomes
\[
(r+1)r(r-1)/6 \le r(r-1)+2 ,
\]
which is wrong for $r=6,7,8$.
The same  argument applies to all  representations in I), except $(\GR{A}{2},3\varpi_1)$.
(The $SL_3$-case is explained in Remark~\ref{rmk-SL3}.)

\textbullet \ \ For list II), the inequality of  
Proposition~\ref{prop:upper-bound} is true, and  more accurate estimates
are needed.
\\
Consider the case $(\GR{A}{5},\varpi_3)$.
Here $\dim \sfr(\varpi_3)=20$, $\dim U'=10$ and $U'_\star=\{1\}$.
Hence $\dim \sfr(\varpi_3)\md U'=10$. Assume that $\sfr(\varpi_3)\md U'\simeq \mathbb A^{10}$. The number of basic invariants of degree $1$ equals $a(\varpi_3)=2$. It is known
that $\bbk[\sfr(\varpi_3)]^G$ is generated by a polynomial of degree $4$. This is our third basic invariant.
Since we must have $\sum_{i=1}^{10} \deg f_i \le \dim\sfr(\varpi_3)=20$, the only possibility
is that the other $7$ basic invariants are of degree $2$.
However, $\mathcal S^2(\sfr(\varpi_3))=\sfr(2\varpi_3)\oplus\sfr(\varpi_1{+}\varpi_5)$, which shows
that the number of basic invariants of degree $2$ is at most 
$\dim \sfr(\varpi_1{+}\varpi_5)^{U'}=4$. This contradiction shows that 
$\bbk[\sfr(\varpi_3)]^{U'}$ cannot be polynomial. 
Such an argument also works for $(\GR{C}{3},\varpi_2)$, $(\GR{C}{3},\varpi_3)$, and
$(\GR{D}{7},\varpi_7)$. 

For $(\GR{A}{r},2\varpi_r)$, $r\ge 2$, we argue as follows. Here the algebra of $U$-invariants is polynomial, and the degrees and weights of basic $U$-invariants are 
$(1,2\varpi_1), (2,2\varpi_2), \dots, (r,2\varpi_r)$, $(r+1,0)$
\cite{br83}. Using Theorem~\ref{thm:base-X-U'},
we conclude that $\bbk[\sfr(2\varpi_r)]^{U'}$ can be generated by $3r+1$ polynomials whose degrees
are $1,1,1;2,2,2;\dots;r,r,r;r+1$. This set of polynomials can be reduced somehow 
to a minimal generating system. Here 
$\dim \sfr(2\varpi_r)\md U'=\dim\sfr(2\varpi_r)-\dim U'=2r+1$.
Assume that $ \sfr(2\varpi_r)\md U'\simeq \mathbb A^{2r+1}$. Then we can remove $r$ 
polynomials from the above (non-minimal) generating system such that the sum of degrees 
of  the remaining polynomials is at most $\dim \sfr(2\varpi_r)=(r+1)(r+2)/2$. 
This means that the sum of degrees of the $r$ removed polynomials must be at least
$r(r+1)$. Clearly, this is impossible.

\noindent
$\mathsf{(ii)\Rightarrow (i)}$. All representations in Table~\ref{tab1} have a polynomial algebra
of $U$-invariants whose structure is well-understood. 
Therefore, using Theorem~\ref{thm:base-X-U'} we obtain an
upper bound on the number of generators of $\bbk[\sfr(\lb)]^{U'}$. On the other hand, we can
easily compute $\dim \sfr(\lb)\md U'$.  In many cases, these two numbers coincide, which immediately proves that $\bbk[\sfr(\lb)]^{U'}$ is polynomial.
In the remaining cases, we use a simple procedure that allows us to reduce 
the non-minimal generating system provided by Theorem~\ref{thm:base-X-U'}.
This appears to be sufficient for our purposes.

\textbullet \ \ For $G=\GR{D}{5}$, the algebra $\bbk[\sfr(\varpi_5)]^{U}$ has two
generators whose degrees and weights are $(1,\varpi_4)$ and $(2,\varpi_1)$. By 
Theorem~\ref{thm:base-X-U'},
$\bbk[\sfr(\varpi_5)]^{U'}$ can be generated by polynomials of degrees and weights 
$(1,\varpi_4), (1,\varpi_4{-}\ap_4), (2,\varpi_1), (2,\varpi_1{-}\ap_1)$. 
On the other hand, the monoid
$\mathcal M(\sfr(\varpi_5))$ is generated by $\varpi_1,\varpi_4$. Therefore a generic stabiliser
$U_\star$ is generated by the root unipotent subgroups
$U^{\ap_2}$, $U^{\ap_3}$, and $U^{\ap_5}$ (see Remark~\ref{rmk:gen-stab-U'}).
Hence $\dim U_\star=6$ and $\dim U'_\star=3$. Thus
$\dim \sfr(\varpi_5)\md U'=16-15+3=4$ and the above four polynomials freely generate  
$\bbk[\sfr(\varpi_5)]^{U'}$.

The  same method works for  
$(\GR{A}{r},\varpi_r)$;  $(\GR{A}{r},\varpi_{r-1})$;  $(\GR{B}{r},\varpi_1)$;  
$(\GR{C}{r},\varpi_1)$;  $(\GR{D}{r},\varpi_1)$;  $(\GR{B}{r},\varpi_r)$, $r=3,4$;
$(\GR{E}{6},\varpi_1)$.  

There still remain four cases, where this method yields the number of generators that is {\sl one more\/} than $\dim \sfr(\lb)\md U'$. Therefore, we have to prove that one
of the functions provided by Theorem~\ref{thm:base-X-U'} can safely be removed.
The idea is the following. Suppose that $\bbk[\sfr(\lb)]^U$ contains two basic invariants
of the same fundamental weight $\varpi_i$, say $p_1\sim(d_1,\varpi_i)$, 
$p_2\sim (d_2,\varpi_i)$.
Consider the corresponding $U'$-invariant functions $p_1,q_1,p_2,q_2$,
where $q_j\sim (d_j, \varpi_i-\ap_i)$, $j=1,2$. 
Assuming that $p_j,q_j$ are normalised such that $e_i{\cdot}q_j=p_j$, the polynomial
$p_1q_2-p_2q_1\in \bbk[\sfr(\lb)]$ appears to be  $U$-invariant,
of degree $d_1{+}d_2$ and weight $2\varpi_i{-}\ap_i$.
If we know somehow that there is a unique $U$-invariant of such degree and weight,
then  this $U$-invariant is not required for the minimal generating system
of $\bbk[\sfr(\lb)]^{U'}$. For instance, consider the case $(\GR{F}{4},\varpi_1)$.
According to Brion \cite{br83}, the free generators of $\bbk[\sfr(\varpi_1)]^{U(\GR{F}{4})}$ are
$(1,\varpi_1),(2,\varpi_1),(3,\varpi_2),(2,\un{0}),(3,\un{0})$. Theorem~\ref{thm:base-X-U'}
provides a generating system for $\bbk[\sfr(\varpi_1)]^{U'(\GR{F}{4})}$
that consists of eight polynomials, namely:
\[
  (1,\varpi_1), (1,\varpi_1{-}\ap_1), (2,\varpi_1), 
(2,\varpi_1{-}\ap_1), (3,\varpi_2), (3,\varpi_2{-}\ap_2), (2,\un{0}), (3,\un{0}) .
\]
Here the weight $\varpi_1$ occurs twice and $2\varpi_1-\ap_1=\varpi_2$.
Therefore the polynomial $(3,\varpi_2)$ can be removed form this set.
Since $\dim\sfr(\varpi_1)=26$, $\dim U'=20$, and $\dim U'_\star=1$, we have
$\dim\sfr(\varpi_1)\md U'=7$. The other three cases, where it works, are 
$(\GR{B}{5},\varpi_5)$, $(\GR{D}{6},\varpi_6)$, $(\GR{E}{7},\varpi_1)$.

This completes the proof of Theorem~\ref{thm:classif}. \hfill $\Box$

\begin{rmk} For a $G$-module $V$, let 
$\ed (\bbk[V]^{U'})$ denote 
the {\it embedding dimension\/} of $\bbk[V]^{U'}$, i.e., the minimal number of 
generators. Since $\bbk[V]^{U'}$ is Gorenstein,  $\ed (\bbk[V]^{U'})-\dim V\md {U'}=\hd (\bbk[V]^{U'})$ 
is the homological  dimension of $\bbk[V]^{U'}$ (see \cite{po83}).
The same argument as in the proof of $\mathsf{(ii)\Rightarrow (i)}$ shows that for
$(\GR{C}{3},\varpi_2)$, $(\GR{C}{3},\varpi_3)$, and $(\GR{A}{5},\varpi_3)$,  we have 
$\hd(\bbk[V]^{U'})\le 2$. 
Hence these Gorenstein algebras of  $U'$-invariants are complete intersections.
We can also prove that $\bbk[\sfr(2\varpi_r)]^{U'(\GR{A}{r})}$ is a complete intersection,  of
homological dimension  $r-1$.
This means that {\sl a postreriori} the following is true: 
{\it If $G$ is simple, $V$ is irreducible, and $\bbk[V]^U$ is polynomial, then 
$\bbk[V]^{U'}$ is a complete intersection.}
It would be interesting to realise whether it is true in a more general situation.
\end{rmk}

\begin{rmk}  There is a unique item in Table~\ref{tab1}, where the sum of degrees of the basic invariants  equals $\dim \sfr(\lb)$ or, equivalently, the sum of weights equals
$2\rho-|\Pi|$. This is $(\GR{B}{5},\varpi_5)$. By Theorem~\ref{thm:analog-5-4-26}(iii),
this is also the only case, where the set of points in $\sfr(\lb)$
with non-trivial $U'$-stabiliser does not contain a divisor.
\end{rmk}

\end{document}